\documentclass[12pt,twoside]{amsart}
\usepackage{amssymb,amsfonts, amsmath, mathrsfs}
 \usepackage{paralist, amscd}

\makeatletter
\evensidemargin
\oddsidemargin
\makeatother

\newtheorem{theorem}{Theorem}[section]
\newtheorem{corollary}{Corollary}[section]

\newtheorem{definition}{Definition}[section]

\newtheorem{proposition}{Proposition}[section]
\newtheorem{problem}{Problem}
\newtheorem{conjecture}{Conjecture}
\newtheorem{remark}[theorem]{Remark}

\numberwithin{equation}{section}
\allowdisplaybreaks

\begin{document}

\thispagestyle{empty}
\setcounter{page}{1}

\noindent
$~$ \\ [.3in]

\begin{center}
{\large\bf The approximate fixed point property in Hausdorff topological vector spaces and applications}

\vskip.20in

Cleon S. Barroso$^{1}$ \\[2mm]
{\footnotesize
$^{1}$Departamento de Matem\'atica\\
Universidade Federal do Cear\'a, Campus do Pici, Bl. 914, 60455-760, Brazil\\

{\it e}-mail: cleonbar@mat.ufc.br
}
\end{center}

{\footnotesize
\noindent
{\bf Abstract.}  Let $\mathscr{C}$ be a compact convex subset of a Hausdorff
topological vector space $(\mathcal{E},\tau)$ and $\sigma$ another
Hausdorff vector topology in $\mathcal{E}$. We establish an approximate
fixed point result for sequentially continuous maps $f\colon
(\mathscr{C},\sigma)\to (\mathscr{C},\tau)$. As application, we
obtain the weak-approximate fixed point property for demicontinuous
self-mapping weakly compact convex sets in general Banach spaces and use this to
prove new results in asymptotic fixed point theory. These results
are also applied to study the existence of limiting-weak solutions for differential
equations in reflexive Banach spaces.
{\it The author is supported by CNPq grants.}
}
\vskip.2in

\section{Introduction}\label{sec:intro}

In fixed point theory, one of the main directions of investigation
concerns the study of the fixed point property in topological spaces.
Recall that a topological space $\mathscr{C}$ is said to have the
fixed point property if every continuous map $f\colon
\mathscr{C}\to\mathscr{C}$ has a fixed point. The major contribution to this subject has been provided by Tychonoff. Every  compact convex subset of a locally convex space has the fixed point property. Historically, this celebrated result first appeared in 1911 with Brouwer for finite dimensional spaces, and after in 1930 with Schauder for general Banach spaces. As early as 1935, Tychonoff was able to prove the locally convex case. Since then, it has shown to be an extremely useful tool to prove existence results for a great variety of problems in nonlinear analysis.

In this paper we deal with another important and current branch of fixed point theory, namely the study of the approximate fixed point property. In this topic, the focus is primarily on the possibility of proving the existence of a sequence $\{x\sb n\}$ in $\mathscr{C}$ such that $x\sb n-f(x\sb n)\to 0$. The idea of analysing the approximation of fixed points of a continuous map is not new. Apparently this was first exploited by H. Scarf in his 1967 paper \cite{Scarf}, where a constructive method of computing fixed points of continuous mappings of a simplex into itself was described. Nowaday, the interest in approximate fixed point results arise naturally in the study of some problems in economics and game theory, including for example the Nash equilibrium approximation in games, see \cite{Tijs} and references therein.  It is the purpose of the present paper to study two problems concerning approximate fixed point property on an ambient space with different topologies. The first one can be formulated as follows:

\begin{problem}\label{prob:1} Let $\mathscr{C}$ be a compact
convex subset of a Hausdorff topological vector space
$(\mathcal{E},\tau)$ and $\sigma$ another Hausdorff vector topology in
$\mathcal{E}$. Then, every sequentially continuous mapping $f\colon
(\mathscr{C},\sigma)\to (\mathscr{C},\tau)$ has a $\tau$-approximate
fixed point sequence, that is, a sequence $\{x\sb n\}$ in
$\mathscr{C}$ such that $x\sb n-f(x\sb n)\stackrel{\tau}{\to}
0$.
\end{problem}

We next emphasize similarities with earlier results. In the case when $\sigma=\tau$ some positive answers are already well known. For instance, if
$\mathscr{C}$ is a bounded closed and convex subset of a Banach space
$(\mathcal{E},\|\cdot\|)$, then every non-expanding mapping $f\colon\mathscr{C}\to\mathscr{C}$ has a
strong approximate fixed point sequence, that is to say a sequence $\{x\sb
n\}$ in $\mathscr{C}$ such that $ \lim\sb{n\to\infty}\|x\sb
n-f(x\sb n)\|=0$.  In some cases, this property remains true even without imposing boundedness on the set $\mathscr{C}$, see Reich \cite{Reich3} and the references therein. Dobrowolski \cite{Dobrowolski1} showed that if
$(\mathcal{E},|\cdot|)$ is a separable metric linear space and
$\mathscr{C}$ is a compact convex subset of $\mathcal{E}$ with the
simplicial approximation property, then every $|\cdot|$-continuous mapping
$f\colon \mathscr{C}\to\mathscr{C}$ has an $|\cdot|$-approximate
fixed point sequence. In \cite{Idzik} Idzik established a result on the existence of $\gamma$-almost fixed points for single-valued functions in the context of Hausdorff topological vector spaces.

We note, however, that it is not immediately clear what happens if $\sigma\neq \tau$. In this context, it is worthwhile to remark that $\tau$-convergence in Problem \ref{prob:1}
is the most natural way to
approximate fixed points for $f$. The reason for this is that there are situations where $\sigma$ is finer than $\tau$ and $f$ has no $\sigma$-approximate fixed points. As an example, consider the case where $(\mathcal{E},\|\cdot\|)$ is a Banach space. Recall
that a mapping $f\colon \mathscr{C}\to\mathscr{C}$ is called
demicontinuous if it maps strongly convergent sequences into
weakly convergent sequences. In this case, given a closed convex
subset $\mathscr{C}$  of $\mathcal{E}$ being
noncompact in norm we cannot hope, in general, to obtain
strong approximate fixed-point sequences for $f$.
This basic conclusion is, for instance, a consequence of the following remarkable
result due to Lin and Sternfeld \cite{Lin}.

\begin{theorem}\label{trm:Int1} Let $\mathscr{C}$ be a noncompact
closed and convex subset of a Banach space $\mathcal{E}$. Then there
is a Lipschitz map $f\colon \mathscr{C}\to\mathscr{C}$ without
approximate fixed points; i.e, there exists a positive number
$\delta$ such that $\|x-f(x)\|>\delta$ for all $x\in \mathscr{C}$.
\end{theorem}

Therefore, in view of this important result, it is natural to
formulate the following variant of Problem \ref{prob:1} which is our second source of motivation.

\begin{problem}\label{prob:2} Let $\mathscr{C}$
be a weakly compact convex subset of a Banach space $\mathcal{E}$
and $f\colon \mathscr{C}\to\mathscr{C}$ a demicontinuous mapping.
Then $f$ has a weak-approximate fixed point sequence, that is, a
sequence $\{x\sb n\}$ in $\mathscr{C}$ such that $x\sb n-f(x\sb
n)\rightharpoonup 0$.
\end{problem}

These problems were the starting point for the research on which
this paper is based upon. Unfortunately, we do not know if they have been previously considered in the literature. However, although it seems to us be the first time that they are studied, we would mention that in the context of Problem \ref{prob:2} there is an intimate connection between approximate fixed points and demicontinuous pseudocontractions in Hilber spaces, see \cite[Lemma 1]{Moloney}.

Let us outline briefly the content. In
Section 2 we present the main result of this paper. By assuming
the existence of appropriate seminorms on $\mathcal{E}$ and using
the Schauder method for finding fixed points, in
Theorem \ref{trm:1}, we get a partial solution for Problem \ref{prob:1}. The crucial observation used in the proof is the fact that algebraic isomorphisms between Hausdorff topological vector spaces of the same finite dimension are, in fact, homeomorphisms. In Section 3, we present some of the theoretical implications of this
result. Among other things, we show that
if $\mathcal{E}$ is a Hausdorff locally convex space whose
topological dual space is weak-star separable,
then Problem \ref{prob:1} can be solved when $\mathscr{C}$ is a (weakly) compact convex
subset of $\mathcal{E}$, see Corollary \ref{cor:3}. By using this fact, we then are able to give a complete solution of
Problem \ref{prob:2} (cf. Theorem \ref{trm:SolPb2}). As a consequence, we provide a partial solution of a
conjecture due Steinlein in connection with the asymptotic fixed point theory in
reflexive Banach spaces. As an example to illustrate how these results can be applied in practice, in the last section we consider the problem of finding limiting-weak solutions for a class of ordinary differential equations in reflexive Banach spaces. Such equations are closely related to Peano's theorem in infinite dimensional spaces.

\section{Approximate Fixed Point Property}
In this section we shall establish our main result on approximate fixed points. As mentioned
earlier, it provides a partial solution for the Problem
\ref{prob:1}. Before going into details, let us introduce first some notation and basic definitions which we shall use in the sequel. Throughout this section we denote by $(\mathcal{E},\tau)$ a Hausdorff
topological vector space, and by $\mathscr{C}$ a nonempty subset of
$\mathcal{E}$.

\begin{definition}\label{def:1} A mapping $f$ of a topological space $X$ into a topological
space $Y$ is said to be sequentially continuous if, for every sequence $\{x\sb n\}$
which converges to a point $x$ in $X$, the sequence $\{f(x\sb n)\}$ converges to $f(x)$ in $Y$.
\end{definition}

It is clear that if $f$ is continuous, then it is sequentially
continuous.

\begin{definition} Let $f\colon \mathscr{C}\to\mathscr{C}$ be a
mapping. A sequence $\{x\sb n\}$ in $\mathscr{C}$ is called a
$\tau$-approximate fixed point sequence for $f$ if $x\sb n-f(x\sb
n)\stackrel{\tau}{\to} 0$, as $n\to\infty$.
\end{definition}

Hereafter, we will say that $\mathscr{C}$ has the
$\tau$-approximate fixed point property if, whenever we take another Hausdorff vector topology $\sigma$ in $\mathcal{E}$,
then every sequentially continuous mapping
$f\colon(\mathscr{C},\sigma)\to(\mathscr{C},\tau)$ has a
$\tau$-approximate fixed point sequence. For the sake of simplicity, we shall use the term "$\tau$-afp
property" to refer to sets with this property. Analogously, $(\mathcal{E},\tau)$ is said to have the $\tau$-afp
property if every compact convex subset $\mathscr{C}$ of $\mathcal{E}$
has the $\tau$-afp property. The following definition will be of central importance in our study of approximate fixed points.

\begin{definition} An admissible function for $\mathscr{C}$ on $\mathcal{E}$ is an extended real-valued function $\rho\colon \mathcal{E}\to [0,\infty]$ such that
 \begin{compactenum}[(i)]

\item The mapping $(x,y)\mapsto \rho(x-y)$ is continuous on $\mathscr{C}\times \mathscr{C}$,

\item $\rho(x+y)\leq \rho(x)+\rho(y)$, for all $x,y\in \mathcal{E}$,

\item $\rho(\lambda x)=|\lambda| \rho(x)$, for all $\lambda \in \mathbb{R}$ and $x\in \mathcal{E}$,

\item If $x,y\in\mathscr{C}$ and $\rho(x-y)=0$, then $x=y$.

\end{compactenum}
\end{definition}

\begin{remark} Notice that if $\rho$ is an admissible function for $\mathscr{C}$ on $\mathcal{E}$, then it defines a metric on $\mathscr{C}$  whose induced topology is coarser than $\tau$.
\end{remark}

It is instructive to compare the notion of continuity in the sense of (i) with the usual one. It is easy to see that if $\rho$ is continuous on $\mathcal{E}$, then $(x,y)\mapsto \rho(x-y)$ is continuous on $\mathscr{C}\times\mathscr{C}$. Furthermore, if (i)-(iii) hold then $\rho$ is continuous on $\mathscr{C}$. It is not true, in general, that if $\rho$ is continuous on $\mathscr{C}$, then it satisfies (i). For example, if $\mathcal{E}=\mathbb{R}$ and $\mathscr{C}=[0,\infty)$, then the mapping $\rho\colon \mathbb{R}\to [0,\infty]$ defined by
\begin{equation*}
\rho(x)=\left\{
\begin{split}
1/x,\   &\, \text{ if }\ x>0,\\
\infty,\ &\, \text{ if }\ x=0,\\
0,\      &\, \text{ if }\ x<0,
\end{split}\right.
\end{equation*}
is continuous on $\mathscr{C}$. However, the mapping $F\colon \mathscr{C}\times\mathscr{C}\to [0,\infty]$ given by $F(x,y)=\rho(x-y)$ is not continuous at the point $(1,1)$. Indeed, it suffices to see that $(1-1/k,1)$ converges to $(1,1)$ in $\mathscr{C}\times\mathscr{C}$, while that $F(1-1/k,1)=0$ and $F(1,1)=\infty$.

It is important to highlight also that the advantage of considering admissible functions as above, is the possibility of working with extended real seminorms in topological vector spaces. As we shall see, this will be fundamental in our approach. However, this deserves a comment. Although both assumptions (ii) and (iii) enable the correct use of the Schauder's method for finding fixed points, admissible functions could perfectly not be useful for getting finite dimensional approximations by finite rank continuous operators, if they were not compatible with assumptions of continuousness. In other words, although it may seem quite reasonable, it is not immediately evident that this class of admissible functions is sufficiently good to imply that the Schauder-projection operator is continuous. This certainly would not allow us to use the Brouwer's fixed point result in a final step in getting fixed points. The following proposition settles this question quite nicely.

\begin{proposition}\label{prop:SchProj} Let $\rho$ be an admissible function for $\mathscr{C}$ on $\mathcal{E}$. Then for any $\epsilon>0$ and $p\in \mathscr{C}$, the function $g\colon \mathscr{C}\to [0,\infty)$ given by
$$
g(x)=\max\{\epsilon-\rho(x-p),0\},
$$
is continuous on $\mathscr{C}$.
\end{proposition}

\begin{proof} Firstly, let us recall that the effective domain of $\rho$ is the set
$$
\mathscr{D}(\rho)=\{ x\in \mathcal{E}\colon \rho(x)<\infty\}.
$$
Let $x\sb 0$ be a point in $\mathscr{C}$ and $\delta>0$ be arbitrary. By assumption, there exists a neighborhood $U\times V$ of $(x\sb 0,p)$ in $\mathscr{C}\times\mathscr{C}$ such that
$$
\rho(x\sb 0-p)-\delta\leq \rho(x-z)\leq \rho(x\sb 0-p)+\delta,
$$
for all $(x,z)\in U\times V$. If $x\sb 0-p\not\in \mathscr{D}(\rho)$ then $\rho(x\sb 0-p)=\infty$ and, hence, $\rho(x-p)=\infty$ for all $x$ in $U$. In consequence, $g(x)=g(x\sb 0)=0$ for all $x\in U$. In case $x\sb 0-p\in \mathscr{D}(\rho)$, we can conclude that $x-p\in \mathscr{D}(\rho)$ for all $x$ in $U$. In this case, it is easy to see that $g(x\sb 0)+\delta\geq g(x)$, for all $x\in U$. On the other hand, if $g(x\sb 0)=0$, then clearly $g(x)\geq g(x\sb 0)-\delta$ holds for every $x\in U$.  Assuming now that $g(x\sb 0)=\epsilon-\rho(x\sb 0-p)$, we have $g(x\sb 0)-\delta\leq \epsilon-\rho(x-p)\leq g(x)$, for all $x$ in $U$. In any case, we have proven that $g$ is continuous at $x\sb 0$, and hence continuous in $\mathscr{C}$. The proof is complete.
\end{proof}

We are now ready to establish the main result of this paper.

\begin{theorem}\label{trm:1} Let $\mathscr{C}$ be a compact convex subset of $(\mathcal{E},\tau)$. Assume
that $\mathscr{C}$ has an admissible function on $\mathcal{E}$.
Then $\mathscr{C}$ has the $\tau$-afp property.
\end{theorem}

\begin{proof} Fix any $n\geq 1$. From (i) and the fact that $\rho(0)=0$, it follows that if $x\in\mathscr{C}$ then the set $B\sb{1/n}(x)=\{y\in\mathscr{C}\colon \rho(y-x)<1/n\}$ is $\tau$-open in $\mathscr{C}$ with respect to the relative topology of $\mathcal{E}$. Thus, the family $\{B\sb{1/n}(x)\colon x\in\mathscr{C}\}$ is an open covering of the compact set $\mathscr{C}$. From compactness we can extract a finite subcovering, i.e. a finite subset $\Gamma\sb n\colon=\{x\sb 1,\dots,x\sb{N\sb n}\}$ of $\mathscr{C}$ such that
$$
\mathscr{C}=\bigcup\sb{i=1}\sp{N\sb n} B\sb{1/n}(x\sb i).
$$
Let $P\sb n\colon \mathscr{C}\to {\rm{co}}(\Gamma\sb n)\subset \overline{\rm{co}}(\Gamma\sb n)$ be the Schauder's projection associated to $\Gamma\sb n$ and $\rho$, where $\overline{\rm{co}}(\Gamma\sb n)$ denotes the $\tau$-closure of the convex hull of $\Gamma\sb n$. In view of the foregoing proposition, it follows that $P\sb n$ is $\tau$-continuous. Moreover, by using (ii) and (iii) we see that
$$
\rho(P\sb n(x)-x)<1/n,
$$
for all $x\in \mathscr{C}$. Let now $\sigma$ be another Hausdorff vector topology in $\mathcal{E}$ and
$f\colon (\mathscr{C},\sigma)\to(\mathscr{C},\tau)$ a
sequentially continuous mapping. Then the mapping
$$
P\sb n\circ
f\colon (\overline{{\rm co}}(\Gamma\sb n),\sigma)\to(\overline{\rm co}(\Gamma\sb
n),\tau)$$
is also sequentially continuous. Let us denote by $\mathcal{G}\sb n$ the linear span of $\Gamma\sb n$. Observe that the linear operator $\Phi\colon \mathcal{G}\sb n\to \mathbb{E}$ defined by $\Phi(\sum\alpha\sb i x\sb i)=\sum\alpha\sb i{\rm{e}}\sb i$ is an algebraic isomorphism, where $\{{\rm{e}}\sb i\}$ denotes the canonical basis of the space $\mathbb{E}=(\mathbb{R}\sp{N\sb n},{\rm{eucld}})$. Here the word "{\rm eucld}" indicates the euclidean topology. Thus, if we denote by $\Phi\sb\sigma$ (resp. $\Phi\sb\tau$) the mapping $\Phi$ from $(\mathcal{G}\sb n,\sigma)$ (resp. $(\mathcal{G}\sb n,\tau)$) into $\mathbb{E}$ then, from Corollary 4 in
\cite[pg.14]{Day}, it follows that both these maps are linear homeomorphisms. Hence, setting $K\sb n=\Phi\sb\tau(\overline{\rm{co}}(\Gamma\sb n))$ we see that $K\sb n=\Phi\sb\sigma(\overline{\rm{co}}(\Gamma\sb n))$.
$$
\begin{CD}
(\overline{\rm{co}}(\Gamma\sb n),\sigma) @>P\sb n\circ f>> (\overline{\rm{co}}(\Gamma\sb n),\tau)\\
\Phi\sb\sigma\sp{-1}@AAA     @VVV\Phi\sb\tau\\
(K\sb n,{\rm{eucld}})    @>>>   (K\sb n,{\rm{eucld}})
\end{CD}
$$
According above diagram, $\Phi\sb\tau\circ(P\sb n\circ f)\circ\Phi\sb\sigma\sp{-1}$ is a sequentially continuous mapping from $(K\sb n,{\rm{eucld}})$ into itself. Since $K\sb n$ is convex and compact with respect to the ${\rm{eucld}}$-topology, it follows from Brouwer's fixed point theorem that
$$
[\Phi\sb\tau\circ(P\sb n\circ f)\circ\Phi\sb\sigma\sp{-1}](z\sb n)=z\sb n,
$$
for some $z\sb n\in K\sb n$. Thus $(P\sb n\circ f)(u\sb n)=u\sb n$, where $u\sb n=\Phi\sp{-1}(z\sb n)$. It follows then that
$$
\rho(u\sb n-f(u\sb n))<1/n,
$$
for all $n\geq 1$. Using now the following fact from general topology:
\begin{center}
{\it If $\mathscr{T}\sb 1$, $\mathscr{T}\sb 2$ are Hausdorff topologies on a set $X$ such that $\mathscr{T}\sb 2$ is finer than $\mathscr{T}\sb 1$ and such that $(X,\mathscr{T}\sb 2)$ is compact, then $\mathscr{T}\sb 1=\mathscr{T}\sb 2$},
\end{center}
we can conclude that $\mathscr{C}$ is sequentially compact, for $\tau$ is finer than $\tau\sb\rho$ on $\mathscr{C}$, the metric topology induced by $\rho$. Thus, we may assume (by passing to a subsequence if necessary) that $u\sb n\stackrel{\tau}{\to} x$ and $f(u\sb
n)\stackrel{\tau}{\to} y$, for some $x,y\in \mathscr{C}$. Hence, in view of (i), we get
$$
\rho(x-y)=0,
$$
and so $x=y$ by (iv). This shows that $u\sb n-f(u\sb n)\stackrel{\tau}{\to}0$ and
concludes the proof.
\end{proof}

\section{Theoretical Implications}
In this section, we explore some of the theoretical implications of Theorem \ref{trm:1}. Although at first glance the conditions required in this result seem somewhat restrictive, they turn up in several practice
situations as we shall soon see. Before seeing this, let us first
recall a basic concept. A family of seminorms $\mathfrak{F}$ in a
vector space $\mathcal{E}$ separates points when
$$
\rho(x)=0 \text{ for all }
\rho\in\mathfrak{F} \text{ imply that } x=0.
$$

Our first corollary is as follows.

\begin{corollary}\label{cor:2} Let $\mathscr{C}$ be a compact convex
subset of a Hausdorff topological vector space $(\mathcal{E},\tau)$ and
$\mathfrak{F}$ a countable family of seminorms
on $\mathcal{E}$ which separate points of
$\mathscr{C}-\mathscr{C}$ and such that the topology $\mathcal{T}$
generated by $\mathfrak{F}$ is coarser than $\tau$ in
$\mathscr{C}$. Then $\mathscr{C}$ has the $\tau$-afp property.
\end{corollary}

\begin{proof} We may set $\mathfrak{F}=\{\rho\sb n\colon n\in\mathbb{N}\}$.
Since $\mathscr{C}$ is compact and $\mathcal{T}$ is coarser than $\tau$, each $\rho\sb n$ restricted to
$\mathscr{C}$ is $\tau$-continuous. Thus we have $\max\{\rho\sb n(x)\colon
x\in\mathscr{C}\}<\infty$ for all $n\in\mathbb{N}$. By replacing the
seminorms $\rho\sb n$ by suitable positive multiples, if
necessary, we may assume that
\begin{equation}\label{eqn:cor1}
\max\{\rho\sb n(x)\colon x\in\mathscr{C}\}\leq 2\sp{-n-1},
\end{equation}
for all $n\in\mathbb{N}$. We then define our admissible function $\rho\colon\mathcal{E}\to [0,\infty]$ as
$$
\rho(x)=\sum\sb{n=1}\sp\infty\rho\sb n(x),\quad x\in\mathcal{E}.
$$
Let us check conditions of Theorem \ref{trm:1}. Notice that $\rho(x-y)<\infty$ for all $x,y\in\mathscr{C}$. Moreover, one readily checks (ii)-(iv). Using now (\ref{eqn:cor1}), we see that the sequence of functions $\rho\sp n(x-y)=\sum\sb{i=1}\sp n \rho\sb i(x-y)$ is Cauchy w.r.t. the topology of uniform convergence on $\mathscr{C}\times\mathscr{C}$. Thus $\rho\sp n(x-y)$ converges uniformly on $\mathscr{C}\times \mathscr{C}$ to $\rho(x-y)$. Furthermore, to verify that (i) holds, we have only to ensure this for each $\rho\sb n$. Let $(x\sb\nu,y\sb\nu)$ be a net in $\mathscr{C}\times\mathscr{C}$ converging to $(x,y)$. Since $\tau$ is finer than $\mathcal{T}$ on $\mathscr{C}$, both $\rho\sb n(x\sb\nu-x)$ and $\rho\sb n(y\sb\nu-y)$ converge to $0$. We may then apply triangular inequality to conclude that $|\rho\sb n(x\sb\nu-y\sb\nu)-\rho\sb n(x-y)|\to 0$. As desired. Theorem
\ref{trm:1} now implies that $\mathscr{C}$ has the $\tau$-afp property.
\end{proof}

Assume now that $(\mathcal{E},\tau)$ is a locally convex Hausdorff
topological vector space whose topological dual $\mathcal{E}\sp *$ is weak-star
separable. It is well-known that $\mathcal{E}\sp *$ is total over
$\mathcal{E}$. Then, for a weak-star dense sequence $\{x\sp *\sb n\}$ in
$\mathcal{E}\sp *$, it follows that
$$
x\mapsto |x\sp *\sb n(x)|
$$
yields a countable family $\mathfrak{F}$ of $\tau$-continuous (resp. weak-continuous) seminorms on $\mathcal{E}$ which separates points. In this case, notice that the topology $\mathcal{T}$ determined by $\mathfrak{F}$ is coarser than $\tau$ (resp. weak topology). Therefore, in view of Corollary \ref{cor:2}, we conclude that every
compact (resp. weakly compact) convex subset of $\mathcal{E}$ has the $\tau$-approximate (resp. weak) fixed point
property.\vspace{.2cm}

This discussion results in the following statement.

\begin{corollary}\label{cor:3} Every (weakly) compact convex subset $\mathscr{C}$
of a Hausdorff locally convex space $(\mathcal{E},\tau)$ whose
topological dual space $\mathcal{E}\sp *$ is weak-star separable has
the (weak) $\tau$-afp property.
\end{corollary}

The preceding discussion still deserves a comment. Notice that on a compact set $\mathscr{C}$ in a Hausdorff locally convex space $\mathcal{E}$ the weak topology coincides with the initial one. More generally, it coincides with every topology on $\mathscr{C}$ generated by any countable family of continuous linear functionals separating the points in $\mathscr{C}$, (cf. \cite[pp. 364--365]{Bogachev}). It should be noted, however, that in general $\mathcal{T}$ need not generate the weak topology on the whole space $\mathcal{E}$. Indeed, if this were the case, then from Corollary \ref{cor:2} we would be able to conclude that $\mathcal{T}$ coincides with the weak topology on every compact subset of $\mathcal{E}$, which in turn certainly would imply that $\tau$ is identical to the weak topology of $\mathcal{E}$. Therefore, the aforementioned situation holds if and only if the original topology of $\mathcal{E}$ is identical with its weak topology. In mathematical terms, we have:

\begin{proposition} Let $(\mathcal{E},\tau)$ be a locally convex Hausdorff topological vector space and $\mathfrak{F}=\{x\sb n\sp *\}$ a weak-star dense sequence in $\mathcal{E}\sp *$. Suppose that $\mathcal{T}$ is the locally convex topology determined by $\mathfrak{F}$. Then $\mathcal{T}$ coincides with the weak topology on every compact subset of $\mathcal{E}$ if, and only if, $\tau$ is identical to the weak topology of $\mathcal{E}$.
\end{proposition}

\begin{remark} As a consequence of the preceding corollary it follows that every separable Banach space has the weak-afp property. Indeed, if $\mathcal{E}$ is a separable Banach space then, by the Banach-Alaoglu-Bourbaki theorem, each dual ball $B\sb{\mathcal{E}\sp *}(0,n)$ centered at the origin with radius $n$, $n\geq 1$, is a weak-star compact metric space and hence a separable metric space. This implies that the dual $\mathcal{E}\sp *$ is weak-star separable.
\end{remark}

\begin{remark} Let $(\mathcal{E},\tau)$ be a locally convex space with a Schauder basis. Then, as a simple computation shows, the topological dual space of $\mathcal{E}$ is weak-star separable.
\end{remark}

In view of Lin-Sternfeld's theorem the following result is quite surprising.

\begin{theorem}\label{trm:SolPb2} Let $\mathscr{C}$ be a weakly compact convex subset
of a Banach space $\mathcal{E}$. Then every demicontinuous mapping
$f\colon \mathscr{C}\to\mathscr{C}$ has a weak-approximate fixed point sequence.
\end{theorem}

\begin{proof} Let us denote by $\|\cdot\|$ the norm in $\mathcal{E}$. Without loss of generality,
we may assume that $f$ is
fixed point free and $\mathcal{E}$ is not separable. Pick any $a\in \mathscr{C}$ and denote by $\mathcal{O}(a)$ the orbit
of $a$ under $f$. Now, we construct inductively a sequence $(\mathscr{A}\sb n)$ of closed convex subsets of $\mathscr{C}$ as follows. We set $\mathscr{A}\sb 0=\overline{{\rm{co}}}\mathcal{O}(a)$ and if $n\geq 1$ we put $\mathscr{A}\sb{n+1}=\overline{\rm{co}}(f(\mathscr{A}\sb n))$, where the overline denotes the closure w.r.t. the norm $\|\cdot\|$. It is easily verified that
$$
\mathcal{O}(f\sp{n+1}(a))\subseteq f(\mathscr{A}\sb n)\subseteq \mathscr{A}\sb{n+1},
$$
for all $n\geq 1$. We claim now that each $\mathscr{A}\sb n$ is separable. This is evident if $n=0$ since the closed linear span of $\mathcal{O}(a)$ is a separable Banach subspace of $\mathcal{E}$. By induction on $n$, and by the fact that $f$ is demicontinuous together with Mazur's theorem, we conclude that if $\mathscr{A}\sb n\subseteq\overline{\{ x\sb k\sp n\colon k\geq 1\}}$ for some $\{x\sb k\sp n\colon k\geq 1\}\subset \mathscr{A}\sb n$, then $f(\mathscr{A}\sb n)\subset \overline{\rm{co}}(f(x\sb k\sp n)\colon k\geq 1)$. This completes the proof of our claim. As a consequence, if we set $\mathscr{B}\sb k=\displaystyle\cap\sb{n=k}\sp\infty\mathscr{A}\sb n$, then the following closed
convex subset of $\mathscr{C}$
$$
\mathscr{D}=\overline{\displaystyle\cup\sb{k=0}\sp\infty\mathscr{B}\sb k},
$$
must be separable too. Notice that, since $\mathscr{C}$ is weakly compact, each $\mathscr{B}\sb k$
is nonempty. Moreover, it is easy to see that $f(\mathscr{B}\sb k)\subseteq \mathscr{B}\sb{k+1}$, for all $k\geq 1$. Hence, using again the fact that
$f$ is demicontinuous, we see that $\mathscr{D}$
is invariant under $f$. Finally, since $\mathscr{D}\subset \overline{\rm{span}}(\{d\sb j\colon j\geq 1\})$ for some dense sequence $\{d\sb j\}$ in $\mathscr{D}$, we reach the conclusion of theorem
by means of Corollary \ref{cor:3}.
\end{proof}

\begin{remark} Notice that Theorem \ref{trm:SolPb2} answers completely the question stated in Problem \ref{prob:2}.
\end{remark}

A very interesting problem in fixed point theory
consists of solving the Schauder conjecture for the class of nonexpansive mappings. That is, to show that every reflexive Banach space has the fixed point
property for nonexpansive self-mapping bounded closed convex
sets. This problem has been studied extensively since 1965 when Browder
\cite{Browder} and G\"ohde \cite{Gohde} independently proved that every
uniformly convex space has the aforementioned property, see for example \cite{Bailon,Kirk} and also the references
\cite{Dowling,Maurey,Sine, Smart} for other related results. Here, as a direct application of Theorem \ref{trm:SolPb2} we obtain the
following fixed point result for continuous maps in general Banach
spaces.

\begin{corollary}\label{cor:fp1} Let $\mathscr{C}$ be a weakly
compact convex subset of a Banach space $\mathcal{E}$ and $f\colon
\mathscr{C}\to\mathscr{C}$ a continuous mapping. Suppose that
$(I-f)(\mathscr{C})$ is sequentially weakly closed. Then $f$ has
fixed point.
\end{corollary}

\begin{proof} By Theorem \ref{trm:SolPb2}, there exists a sequence
$\{x\sb n\}$ in $\mathscr{C}$ such that $x\sb n-f(x\sb
n)\rightharpoonup 0$. By assumption, we get $0\in
(I-f)(\mathscr{C})$ and so $f(x)=x$, for some $x\in \mathscr{C}$.
This completes the proof.
\end{proof}

\begin{remark} It is worthwhile to note that in view of Theorem \ref{trm:Int1}, the sequential weak closedness assumption in Corollary
\ref{cor:fp1} can not be removed.
\end{remark}

Another fruitful field of research in fixed point theory concentrate
efforts in the direction of obtaining consistent results from
conditions imposed upon the iterates $f\sp m$ for $m$ sufficiently large.
Jones \cite{Jones1} introduced the term
 "asymptotic fixed point theorems" to describe such results.
One of the reasons for the usefulness of such
theorems lies in the fact that they are intrinsically related to the problem of finding periodic solutions
of ordinary differential equations, differential-difference equations, and
functional differential equations, see for instance
\cite{Halanay,Jones1,Jones2,Jones3,Yoshizawa}.\vspace{.1cm}

In \cite{Steinlein} Heinrich Steinlein formulated the following conjecture:
\begin{conjecture}
Let $\mathcal{E}$ be a Banach space, $\mathscr{C}\subset \mathcal{E}$ a
nonempty closed bounded convex set and $f\colon \mathscr{C}\to\mathscr{C}$
be a continuous map such that $f\sp m$ is compact for some $m\in\mathbb{N}$.
Then, $f$ has a fixed point.
\end{conjecture}

Partial solutions for this problem have been given, for instance, by Browder \cite{Browder2},
Nussbaum \cite{Nussbaum} and Steinlein \cite{Steinlein}. Our next result proves
this conjecture for the case when $f\sp m$ is strongly continuous for some $m$, i.e. $f\sp m(x\sb n)\to f\sp m(x)$ whenever that
$x\sb n\rightharpoonup x$. Note that in reflexive spaces strongly continuous maps are compact.

\begin{corollary} Let $\mathscr{C}$ be a weakly compact convex subset of a Banach space $\mathcal{E}$,
and let $f\colon\mathscr{C}\to\mathscr{C}$ be a demicontinuous mapping. Suppose that $f\sp m$ is
strongly continuous for some $m\in\mathbb{N}$. Then $f$ has a fixed point.
\end{corollary}

\begin{proof} Let $\{x\sb n\}$ be a weak-approximate fixed point sequence for $f$, for example the one given in
Theorem \ref{trm:SolPb2}. From Eberlein-\v{S}mulian's theorem we conclude that up to a subsequence, still denoted by $x\sb n$, $x\sb n\rightharpoonup x$ in $\mathscr{C}$
for some $x\in \mathscr{C}$. In particular, $f(x\sb n)\rightharpoonup x$. Since $f\sp m$ is strongly continuous,
this implies that $f(p)=p$ where $p=f\sp m(x)$.
\end{proof}




\section{On Differential Equations in Reflexive Spaces}
In this section, we are concerned with the following vector-valued differential equation:
\begin{eqnarray}\label{prob:P}
\left\{
\begin{split}
&u\sb t = f(t,u)\ \text{ in } E, \\
&u(0)= u\sb 0\in E,
\end{split}\right.
\end{eqnarray}
where $t\in I=[0,T]$, $T>0$, $E$ is a reflexive Banach space and $f\colon I\times E\to E$. Here, the field
$f$ is assumed to be a Caracth\'eodory mapping, that is,
\begin{compactenum}

\item[($f\sb 1$)] for all $t\in I$, $f(t,\cdot)\colon E\to E$ is continuous,

\item[($f\sb 2$)] for all $x\in E$, $f(\cdot, x)\colon E\to E$ is measurable.

\end{compactenum}
\vspace{.2cm}

Differential equations in abstract spaces have been the object of thorough and fruitful study in many works. In \cite{Diudonne}
Diudonn\'e constructed an example of a continuous mapping $f\colon c\sb 0\to c\sb 0$  for
which (\ref{prob:P}) has no solution. Naturally, this leads us to suspect that the strong continuity is not, in general, the right assumption to solve this problem. In spite of this, Godunov \cite{Godunov} proved that for every infinite dimensional Banach
space there exists a continuous field $f$ such that (\ref{prob:P}) has no solution. Since then,
several approaches have been developed to establish the so called Peano's property
in more general spaces. One of them,
consists in considering the notion of continuity to the setting of
locally convex spaces, see \cite{Szep, Teixeira} and the references therein. In this
section, we explore another approach to (\ref{prob:P}). The basic idea is to weaken the notion
of solution in a way that allows us to derive general existence results even without
having additional conditions of continuity other than $(f\sb 1)$. To this aim, the
theory on weak-approximate fixed points for continuous mappings developed in previous section
will be invoked.
\vspace{.2cm}

We introduce the following notion of weak-approximate solution for (\ref{prob:P}).

\begin{definition}\label{def:5.1}(Limiting weak solutions.) We say that an $E$-valued
function $u\colon I\to E$ is a limiting-weak solution
to the problem (\ref{prob:P}) if $u\in C(I,E)$ and there exists a sequence $(u\sb n)$ in $C(I,E)$ such that\vspace{.1cm}

 \begin{compactenum}[(a)]

 \item $u\sb n\rightharpoonup u$ in $C(I,E)$,

 \item For each $t\in I$,
  $$
  u\sb 0+\int\sb 0\sp t f(s,u\sb n(s))ds\rightharpoonup u(t)\,\,\text{ in }\, E,
  $$

 \item  and, $u$ is almost everywhere strongly differentiable in $I$.

\end{compactenum}
\end{definition}

\begin{remark} The above integral is understood in Bochner sense.
\end{remark}

Of course, Definition \ref{def:5.1} was inspired by Theorem \ref{trm:SolPb2},
ensuring the existence of weak-approximate fixed point sequences for continuous maps. In our next result
we shall use this theorem to get an existence result of limiting-weak solutions to (\ref{prob:P}).

\begin{theorem}\label{trm:deq}
Let $E$ be a reflexive Banach space and $f\colon I\times E\to E$ be a Carath\'eodory mapping
satisfying
\begin{equation}\label{eqn:deq}
\|f(s,x)\|\leq \alpha(s) \varphi(\|x\|\sb E),\quad \text{ for a.e. } s\in I, \text{ and all } x\in E,
\end{equation}
where $\|\cdot\|\sb E$ denotes the norm of $E$, $\alpha\in L\sb p[0,T]$ for some $1<p<\infty$, and $\varphi\colon [0,\infty)\to (0,\infty)$ is nondecreasing
continuous function such that
$$
\int\sb 0\sp T \alpha(s)ds <\int\sb 0\sp\infty \frac{ds}{\varphi(s)}.
$$
Then, (\ref{prob:P}) has a limiting-weak solution.
\end{theorem}

For the proof of theorem we will rely on the following weak-compactness result of Dunford.

\begin{theorem}[Dundord]\label{trm:Dunford}  Let $(\Omega,\Sigma,\mu)$ be a finite measure space and $E$ be a Banach space such that both $E$ and $E\sp *$ have the Radon-Nikod\'ym property. A subset $\mathscr{C}$ of $L\sb 1(\mu,E)$ is relatively weakly compact if

\begin{compactenum}[(a)]

\item $\mathscr{C}$ is bounded,

\item $\mathscr{C}$ is uniformly integrable, and

\item for each $\Lambda\in \Sigma$, the set $\{\int\sb\Lambda ud\mu\colon u\in \mathscr{C}\}$ is relatively weakly compact.
\end{compactenum}
\end{theorem}

\begin{remark} Notice that every reflexive space has the Radon-Nikod\'ym property. In particular, both $E$ and $E\sp *$ have this property if $E$ is reflexive.
\end{remark}

\begin{proof}[Proof of Theorem \ref{trm:deq}] Let us consider $(\Omega,\Sigma,\mu)$ the usual Lebesgue measure space on $I$ and denote by $L\sb 1(I,E)$ the standard Banach space of all equivalence classes of $E$-valued Bochner integrable functions $u$ defined on $I$ equipped with its usual norm $\|\cdot\|\sb 1$. In what follows we shall use the following notations
\begin{equation*}
\begin{split}
&\mathcal{A}=\{u\in L\sb 1(I,E)\colon \|u(t)\|\sb E\leq b(t) \text{ for a.e. } t\in I\},\\ \\
&\mathcal{B}=\{ v\in L\sb 1(I,E)\colon \|v(t)\|\sb E\leq \alpha(t)\varphi(b(t))\text{ for a.e. } t\in I\},
\end{split}
\end{equation*}
where
$$
b(t)=J\sp{-1}\Big(\int\sb 0\sp t\alpha(s)ds\Big) \quad \text{ and }
\quad J(z)=\int\sb{\|u\sb 0\|\sb E}\sp z\frac{1}{\varphi(s)}ds.
$$
A straightforward computation shows that both $\mathcal{A}$ and $\mathcal{B}$ are convex. Also, as is readily seen, $\mathcal{A}$ is closed in $L\sb 1(I,E)$. Moreover, since $E$ is a reflexive space, we can apply Dunford's theorem to conclude that $\mathcal{B}$ is a relatively weakly compact set in $L\sb 1(I,E)$.
Let us consider now the set
$$
\mathscr{C}=\{ u\in \mathcal{A}\colon u(t)=u\sb 0+\int\sb 0\sp t \overline{u}(s)ds \text{ for a.e. } t\in I, \text{ and some } \overline{u} \in\mathcal{B}\}.
$$
It is easy to see that $\mathscr{C}$ is nonempty and convex. We claim now that $\mathscr{C}$ is closed. Indeed, let $\{ u\sb n\}$ be a sequence in $\mathscr{C}$ such that $u\sb n\to u$ in $L\sb 1(I,E)$. Then
$$
u\sb n(t)\rightarrow u(t)\,\, \text{  in } E
$$
for a.e. $t\in I$.
In particular,
$u\in \mathcal{A}$. On the other hand, since $\mathcal{B}$ is sequentially weakly compact and
$$
u\sb n(t)=u\sb 0+\int\sb 0\sp t \overline{u}\sb n(s)ds\,\, \text{ for a.e. } t\in I,
$$
with $\overline{u}\sb n\in \mathcal{B}$, $n\geq 1$, we may assume that $\{\overline{u}\sb n\}$ converges weakly to some $\overline{u}\in L\sb 1(I,E)$. Then, by fixing any $\phi\in E\sp *$ and taking into account that each $\int\sb 0\sp t \langle \phi,\cdot\rangle ds$ defines a bounded linear functional on $L\sb 1(I,E)$, it follows that
$$
\langle\phi,u(t)-u\sb 0\rangle=\lim\sb{n\to\infty}\langle \phi,\int\sb 0\sp t \overline{u}\sb n(s)ds\rangle=\lim\sb{n\to\infty}\int\sb 0\sp t\langle \phi,\overline{u}\sb n(s)\rangle ds= \int\sb 0\sp t\langle\phi,\overline{u}(s)\rangle ds,
$$
for a.e. $t\in I$. Hence $\langle \phi,u(t)-u\sb 0\rangle=\langle \phi,\int\sb 0\sp t \overline{u}(s)ds\rangle$ for a.e. $t\in I$. This implies that
$$
u(t)=u\sb 0+\int\sb 0\sp t \overline{u}(s)ds,\quad \text{ for a.e. } t\in I,
$$
since $\phi$ was arbitrary.
It remains to show that $\overline{u}\in \mathcal{B}$. To this end, it suffices to apply Mazur's theorem since $\mathcal{B}$ is closed in $L\sb 1(I,E)$ and $\overline{u}\sb n\rightharpoonup u$ in $L\sb 1(I,E)$. This concludes the proof that $\mathscr{C}$ is closed.

Thus, by applying once more Dunford's theorem, we reach the conclusion that $\mathscr{C}$ is weakly compact in $L\sb 1(I,E)$. Let us define now a mapping $F\colon \mathscr{C}\to \mathscr{C}$ by
$$
F(u)(t)=u\sb 0+\int\sb 0\sp t f(s,u(s))ds.
$$
From now on, our strategy will be to obtain a weak-approximate fixed point sequence for $F$ in $L\sb 1(I,E)$ and then deduce that it is itself a weak-approximation of fixed points for $F$ in $W\sp{1,p}(I,E)$, the Sobolev space consisting of all $u\in L\sb p(I,E)$ such that $u'$ exists in the weak sense and belongs to $L\sb p(I,E)$. After this we will use the fact that the embedding $W\sp{1,p}(I,E)\hookrightarrow C(I,E)$ is continuous to recover the corresponding weak convergence in $C(I,E)$.

By using $(f\sb 1)$-$(f\sb 2)$, we see that $F$ is well-defined and that it is continuous with respect to the norm-topology of $L\sb 1(I,E)$. The last assertion follows easily from Lebesgue's theorem on dominated convergence. According to Theorem \ref{trm:SolPb2}, there exists a sequence $\{u\sb n\}$ in $\mathscr{C}$ so that $u\sb n-F(u\sb n)\rightharpoonup 0$ in $L\sb 1(I,E)$. Observe that $u\sb n\in C(I,E)$ for all $n\geq 1$. Moreover, up to subsequences, we may assume that both $\{u\sb n\}$ and $\{F(u\sb n)\}$ converge in the weak topology of $L\sb 1(I,E)$ to some $u$ in $\mathscr{C}$. We claim now that $u\sb n-F(u\sb n)\rightharpoonup 0$ in $C(I,E)$. Before proving this, let us make a pause to get a priori $L\sb p$-estimates for arbitrary functions  $u\in\mathscr{C}$.

\subsection{$L\sb p(I,E)$-Estimates.}  Fix any $u\in \mathscr{C}$:\vspace{.1cm}

\begin{compactenum}[(1)]

\item Using $(f\sb 2)$ we have
\begin{eqnarray}\nonumber
\|F(u)\|\sb{L\sb p}&\leq& \|u\sb 0\|\sb E |I|\sp{1/p}+\Big\{\int\sb 0\sp T\Big\|\int\sb 0\sp tf(s,u(s))ds\Big\|\sb E\sp pdt\Big\}\sp{1/p}\\\nonumber
&\leq & \|u\sb 0\|\sb E |I|\sp{1/p}+\Big\{\int\sb 0\sp T\Big(\int\sb 0\sp t\|f(s,u(s))\|\sb Eds\Big)\sp pdt\Big\}\sp{1/p}\\\nonumber
&\leq &\|u\sb 0\|\sb E |I|\sp{1/p}+\Big\{\int\sb 0\sp T\Big(\int\sb 0\sp t\alpha(s)\varphi(b(s))ds\Big)\sp p dt\Big\}\sp{1/p}\\\nonumber
&\leq & \|u\sb 0\|\sb E|I|\sp{1/p}+\|\alpha\|\sb{L\sb 1[0,T]}\varphi(\|b\|\sb\infty)T\sp{1/p},
\end{eqnarray}
where $\|b\|\sb\infty$ denotes the supremum norm of $b$ on $I$.

\item
Analogously, one can shows that
\begin{eqnarray}\nonumber
\|u\|\sb{L\sb p}\leq \|u\sb 0\|\sb E |I|\sp{1/p}+\|\alpha\|\sb{L\sb 1[0,T]}\varphi(\|b\|\sb \infty) T\sp{1/p}.
\end{eqnarray}

\item It follows now from $(f\sb 2)$ and the $L\sb p$-assumption on $\alpha$ that
\begin{eqnarray}
\|\partial\sb t F(u)\|\sb{L\sb p}\leq \varphi(\|b\|\sb \infty)\|\alpha\|\sb{L\sb p[0,T]},
\end{eqnarray}
and
\begin{eqnarray}
\|\partial\sb t u\|\sb{L\sb p}\leq\varphi(\|b\|\sb \infty)\|\alpha\|\sb{L\sb p[0,T]}.
\end{eqnarray}
\end{compactenum}
In consequence, the above estimates show that both $\{u\sb n\}$ and $\{F(u\sb n)\}$ are bounded sequences in $W\sp{1,p}(I,E)$. In view of the reflexivity of $W\sp{1,p}(I,E)$, by passing to a subsequence, if necessary, we can find $v,w\in W\sp{1,p}(I,E)$ such that $u\sb n\rightharpoonup v$ and $F(u\sb n)\rightharpoonup w$ in $W\sp{1,p}(I,E)$. In particular, $u=v=w$ since the embedding $W\sp{1,p}(I,E)\hookrightarrow L\sb 1(I,E)$ is continuous. Thus, $u\sb n-F(u\sb n)\rightharpoonup 0$ in $W\sp{1,p}(I,E)$. On the other hand, using now the fact the embedding $W\sp{1,p}(I,E)\hookrightarrow C(I,E)$ is also continuous, it follows that
\begin{eqnarray}
u\sb n-F(u\sb n)\rightharpoonup 0 &\text{ in }& C(I,E),\,\,{ and }\\
u\sb n\rightharpoonup u &\text{ in }& C(I,E).
\end{eqnarray}
Therefore
\begin{equation}\label{eqn:1deq}
u\sb 0+\int\sb 0 \sp t f(s,u\sb n(s))ds\rightharpoonup u(t)\,\,\text{ in }\, E,
\end{equation}
for all $t\in I$, which proves $(a)$ and $(b)$ of Definition \ref{def:5.1}. It remains to prove the optimal regularity of the limiting-weak solution $u$. To this end, we may apply again Dunford's theorem to conclude that
$$
K=\{ f(\cdot, u\sb n(\cdot))\colon n\in\mathbb{N}\}
$$
is relatively weakly compact in $L\sb 1(I,E)$. Hence,
by passing to a subsequence if necessary, we get
\begin{equation}\label{eqn:2deq}
\int\sb 0\sp t f(s,u\sb n(s))ds\rightharpoonup \int\sb 0\sp t v(s)ds\,\, \text{ in } E,
\end{equation}
for all $t\in I$ and some $v\in L\sb 1(I,E)$. Combining (\ref{eqn:1deq}) and (\ref{eqn:2deq})
 it follows that
$$
u(t)=u\sb 0+\int\sb 0\sp t v(s)ds,
$$
for all $t\in I$. Hence, following the same arguments as in \cite{Teixeira},
one can prove that $u$ is almost everywhere strongly differentiable in $I$. This completes the proof of Theorem \ref{trm:deq}.
\end{proof}

\begin{remark} It is worthwhile to point out that in view of Godunov's result
Theorem \ref{trm:deq} can be the best possible.
\end{remark}

\section*{Acknowledgements} The author would like to thank the two previous reviewers of the original manuscript. Their constructive suggestions and comments improved significantly the final outcome of this paper. The author would like also to thank Professor Roger D. Nussbaum for discussions and encouragements.

\end{document}